\documentclass[12pt]{article}
\usepackage{amsfonts,amsthm,amsmath,enumerate,hyperref,authblk}
\newtheorem{thm}{Theorem}
\newtheorem{lem}{Lemma}

\newtheorem{mydef}{Definition}

\newtheorem{thmx}{Theorem}

\newtheorem{innercustomthm}{Theorem}
\newenvironment{customthm}[1]
  {\renewcommand\theinnercustomthm{#1}\innercustomthm}
  {\endinnercustomthm}
\title{On the Annihilation of Thin Sets}
\date{\vspace{-8ex}}
\author{Tomer Amit and Alexander Olevskii\footnote{This author's research is supported in part by the Israel Science Foundation.}}
\affil{
T.A.: School of Mathematics, Tel Aviv University \\
Ramat Aviv, 69978 Israel. E-mail: tomer1amit@gmail.com \\
A.O.: School of Mathematics, Tel Aviv University \\
Ramat Aviv, 69978 Israel. E-mail: olevskii@post.tau.ac.il}
\begin{document}
\maketitle
\begin{abstract}
One says that a pair of sets $(S,Q)$ in $\mathbb{R}$ 
                     is 'annihilating'
                     if no function can be concentrated on $S$ while having its Fourier 
                     transform concentrated on $Q$.
                     One uses to distinguish between weak and strong annihilation types.
                     It is well known that if both sets $S$ and $Q$ are of finite measure
                     then they are strongly annihilating.
                     
                     In this paper we prove that if $S$ is a set of finite measure
                     with periodic gaps, and $Q$ is a set of density zero, then
                     weak annihilation holds. 
                     On the other hand a counter-example is constructed, showing
                     that strong annihilation, in general, does not.
\end{abstract}
\section{Introduction}\label{sec:Introduction}
The following definition is inspired by the classical uncertainty principle (for an extensive account of this subject, see \cite{Havin}):
\begin{mydef}
Let $S,Q \subset \mathbb{R}$ be measurable sets. 

\begin{enumerate}[(i)]
	\item Say $(S,Q)$ is \textsl{weakly annihilating} if
\begin{equation*}
\forall f \in L^2(\mathbb{R}), \mathrm{supp}(f) \subset S \text{ and } \mathrm{spec}(f) \subset Q \Longrightarrow f=0 \text{ identically}.
\end{equation*}

\item Say $(S,Q)$ is \textsl{strongly annihilating} if 
\begin{equation*}
\exists C = C(S,Q)>0, \forall f \in L^2(\mathbb{R}), \left\|f\right\|_2^2 \leq C (\left\|f\right\|_{L^2(S^c)}^2 + \left\|\mathcal{F}{f}\right\|_{L^2(Q^c)}^2).
\end{equation*}

\end{enumerate}
\end{mydef}
We use the Fourier transform with the following normalization
\[
\mathcal{F}f(\xi) = \widehat{f}(\xi) = \int\limits_{\mathbb{R}} e^{-2\pi i x \cdot \xi } f(x) \mathrm{d}x.
\]
The support and spectrum of a function on $\mathbb{R}$ are defined respectively as
\[
\mathrm{supp}(f) = \{x \in \mathbb{R} \ | \ f(x) \neq 0\}
\]
and $\mathrm{spec}(f) = \mathrm{supp}(\widehat{f})$. Also, throughout the paper $\left|\cdot\right|$ denotes the Lebesgue measure on $\mathbb{R}$.
\newline
 It is well known that that if both $S$ and $Q$ have finite 
                      Lebesgue measure then $(S,Q)$ is strongly annihilating (\cite{Ben}, \cite{AB}).

                       In \cite{Wolff} a class of sets was defined, so called  '$\epsilon$-thin'
                     sets, which may have infinite measure, but have small density.
                     It was proved in that paper that the strong annihilation property
                     for any pair of sets in this class holds.

                       In the present paper we consider a non-symmetric situation, when
                     $S$ is of finite measure and $Q$ has density zero.                    
\begin{mydef}
We say that a set $Q$ is a set of density zero if 
\[
\left|Q \cap (-r,r)\right| = o(r) \text{ as } r \rightarrow \infty.
\]
\end{mydef}
\begin{mydef}
We say that a set $S$ has \textsl{periodic gaps} if there is an interval $I$ and a number $T>0$ such that
\[
I + kT \subset S^c, \ \forall k \in \mathbb{Z}.
\]
\end{mydef}
We remark that the role of periodic gaps in the uniqueness problem was recently clarified in \cite{Ol2}. \newline
We will prove the following theorems.
\begin{thm}\label{thm:THM1}
Let $S$ be a set of finite measure with periodic gaps, and let $Q$ be a set of density zero. Then the pair $(S,Q)$ is weakly annihilating.
\end{thm}
\begin{thm}\label{thm:THM2}
There is a set $S$ of finite measure with periodic gaps and a set $Q$ of density zero, such that $(S,Q)$ is \textbf{not} strongly annihilating. 
\end{thm}

\section{Payley-Wiener Spaces and Uniqueness Sets}\label{sec:PW}
   The proof of Theorem \ref{thm:THM1} is based on an approach developed in \cite{OU11}.
                      This approach allows one to construct discrete uniqueness sets
                      for Paley-Wiener spaces with (unbounded) spectrum of finite
                      measure.
  Below we present the necessary material from \cite{OU11} (see also \cite{Ol}, Lecture 10).
		\newline
		We let $S \subset \mathbb{R}$ be a set of finite measure, and $F \in L^2(S)$. Recall that the Payley-Wiener space associated to $S$, denoted by $\mathrm{PW}_S$ is defined as the image of $L^2(S)$ under $\mathcal{F}^{-1}$, i.e, functions with Fourier transform 
		supported on $S$. We make the following definition:
  \begin{mydef}   A set $\Lambda$ is a \textsl{uniqueness set (US)} for $\mathrm{PW}_S$ if
\[
                              f \in \mathrm{PW}_S , f|_{\Lambda} = 0  \Rightarrow f=0 \text{ identically}.
\]
                      \end{mydef}
											It is well known that if the spectrum $S$ is bounded, then 
                       any function in the Paley-Wiener space is a trace of an entire function
                       of exponential type. Using this fact one can show that a uniformly discrete US for $\mathrm{PW}_S$ in this case exists.
											\newline
											A set $\Lambda$ is called \textsl{uniformly discrete (u.d.)} if 

                             \[
														\inf\{ |\lambda-\lambda'| \ | \ \lambda \neq \lambda', \lambda,\lambda' \in \Lambda \}>0.
														\]

                        For unbounded spectra of finite measure no analyticity,
                       or even smoothness of $f$ holds (it is merely continuous).
                        Nevertheless discrete uniqueness sets do exist.

                       \begin{thmx}\cite{OU11}\label{thmx:OU}
                       Let $S$ be a set of finite measure. Then there is a u.d. set $\Lambda$,
                       which is a US for PWs. \end{thmx}
                       
                        Assume for simplicity  $\left|S\right|<1$.
                        \begin{mydef} Let  $t \in [0,1]$ and assume that $\exists k \in \mathbb{Z}$ such that $t+k \in S$.
                            The set $A$ of all such $t$ is called call \textsl{the projection of $S$} and is denoted $\mathrm{proj}(S)$.
													\end{mydef}
                        Introduce the $\textsl{multiplicity function}$ as 
												\[
                                   w(t):[0,1]\rightarrow \mathbb{N}, w(t) = \#\{k \text{ in the def. above}\}.
												\]
                        Define the weighted space:
                                 \[
																X= L^2(A;w) := \{ F \ | \  \left\|F\right\|^2 := \int \left|F\right|^2 w dt < \infty \}.
																\]

                        It is easy to see that the dual space $X^{\ast}$ can be identified with
                         $L^2(A,1/w)$ and is embedded in the space $L^1(A)$.

                        The following lemma plays a crucial role in the proof of Theorem \ref{thmx:OU}.

                        \begin{lem}\cite{OU11}\label{lem:MAINLEM} Let $Z_j, \ j=1,2,...$ be pairwise disjoint sets
                        of integers.
                              Assume that for each j the set $Z_j$ is a uniqueness set for
                              the space  
															\[
                                     \widehat{X^{\ast}} := \{ g= \widehat{G} , G \in X^{\ast}\}.
																		\]
 
                              Then the set
                                      \[
																			\Lambda := \bigcup\limits_{j=1}^{\infty}(Z_j + a_j),                           
																			\]

                              is a US for $\mathrm{PW}_S$, whenever $\{a_j\}$  is a sequence of numbers 
                              dense on [0,1]. \end{lem}
															
	\section{Proof of Theorem \ref{thm:THM1}}
		It was noticed in \cite{OU11} that the set $\Lambda$ in Theorem \ref{thmx:OU} can be
    chosen outside of a given set $Q$ of finite measure (This, in particular, implies the result of \cite{Ben}, saying
         that $(S,Q)$ is a weakly annihilating pair). We want to do the same under the conditions of Theorem \ref{thm:THM1}.

		 \begin{customthm}{1'}\label{customthm:MAIN}
 Given $S$ and $Q$ satisfying the conditions of Theorem \ref{thm:THM1}, one can construct a US set for $\mathrm{PW}_S$, disjoint from $Q$.\end{customthm}
                           
		 Observe that in our situation $Q$ is, in general, a set of               
                        infinite measure, so we need an extra argument for the proof.
                        It comes from the 'periodic gaps' condition, which allows us
                        to involve the classic 'Beurling-Malliavin' theorem (see, for instance, \cite{Havin}, Chapter 4). We need the following corollary of that result.

                         \begin{thmx}[Beurling-Malliavin]\label{thmx:BM}
                                 For a given $\sigma > 0$, let $\Gamma$ be a set of integers 
                                 such that for infinitely many dyadic blocks
																	\[
                                          B_k:= [2^k,2^{k+1})  \ , \ k \in \mathbb{N}
																	\]
                                the following condition holds:
																		\[
                                       \# (\Gamma \cap B_k) >  (1-\frac{\sigma}{2}) \left|B_k\right|.      
																	\]
                                Then $\Gamma$ is a US for the space 
																	\[
                                 Y_{\sigma}:=  \{ f = \widehat{F} , F \in L^1([\sigma,1])\}.
																		\]
																		\end{thmx}

                          Combining Lemma \ref{lem:MAINLEM} with Theorem \ref{thmx:BM}, we will now proceed to the proof of
													Theorem \ref{customthm:MAIN}. \newline
                             Let $S$ and $Q$ be given. Consider the set $A := \mathrm{proj}(S)$.
                           It contains a gap which is a proper sub-interval in $[0,1]$. 
                           We can assume it to be $[0,\sigma]$
														\newline
														We first show:
                           \begin{lem}
                          Given  a number $\epsilon > 0$ there is $j_0=j_0(\epsilon)$ 
                                     such that for any $j > j_0$ there is a set $E_j \subset [0,1]$,
                                     $\left|E_j\right|  > 1-\epsilon$, such that for $\alpha \in E_j$,
																	\[
                                      \# [(\mathbb{Z}+\alpha) \cap (B_j \backslash Q)] > (1-\frac{\sigma}{2}) \left|B_j\right|.
																			\]
                       \end{lem}
                
											\begin{proof} For a given $\alpha \in [0,1], r>0$ set
											\[
											G(\alpha,r) = \frac{\#[(\alpha+\mathbb{Z}) \cap Q^c \cap (0,r)]}{r}
											\]
											For every $r>0$, the function $G:\alpha \mapsto G(\alpha,r)$ is measurable and satisfies $0 \leq G \leq 1$. We also have
											\[
											\int\limits_0^1 G(\alpha,r) \mathrm{d}\alpha = \frac{\left|Q^c \cap (0,r)\right|}{r}.
											\]
											
											Since $Q$ is of density $0$, we get
											\[
											\lim\limits_{r \rightarrow \infty} \int\limits_0^1 G(\alpha,r) \mathrm{d}\alpha = 1.
											\]								
											Now fix some $\epsilon>0$. 
											By the last limit, there must be some $r_0$ big enough,
											such that for every $r>r_0$, 
											\[
											\int\limits_0^1 G(\alpha,r) \mathrm{d}\alpha > 1-\frac{\sigma}{4}\epsilon.
											\]
											Let $r>r_0$. Denote 
											\[E_r = \{\alpha \in [0,1] \ | \ G(\alpha,r) > 1-\frac{\sigma}{4}\}
											\]
											Since $0 \leq G \leq 1$, we get
											\[
											1-\frac{\sigma}{4}\epsilon <  \int\limits_0^1 G(\alpha,r) \mathrm{d}\alpha \leq \left|E_r\right| + (1-\left|E_r\right|)(1-\frac{\sigma}{4})
											\]
											which gives
											\[
											\left|E_r\right|>1-\epsilon.
											\]
									
											In particular: for every $\epsilon$, $\exists j_0$ such that $\forall j>j_0$ there exists $E_j$ with $\left|E_j\right|>1-\epsilon$, such that for every $\alpha \in E_j$:
											\[
											\# [(\mathbb{Z}+\alpha) \cap ((0,2^{j+1}) \backslash Q)] > (1-\frac{\sigma}{4}) 2^{j+1} 
											\]
											but since clearly
											\[
											\# [(\mathbb{Z}+\alpha) \cap ((0,2^{j}) \backslash Q)] \leq 2^{j}
											\]
											we get
												\[
                                      \# [(\mathbb{Z}+\alpha) \cap (B_j \backslash Q)] > (1-\frac{\sigma}{4})2^{j+1}-2^{j} = (1-\frac{\sigma}{2}) \left|B_j\right|.
																			\]
											\end{proof}
 As an immediate corollary we have

                          \begin{lem}   Almost every $\alpha \in [0,1]$ satisfies the property:
                                for infinitely many dyadic blocks  $B_j$,             
															\[
																\# \{k \ | \  k+\alpha \in B_j \backslash Q \} > (1- \frac{\sigma}{2}) \left|B_j\right|. 
															\]

  \end{lem}
                              Theorem \ref{thmx:BM} now implies: for a.e. $\alpha$ the set
														\[
                                    \Gamma_{\alpha}:= {(\mathbb{Z}+\alpha) \cap Q^c },
															\]
                              is a US for the space $Y_{\sigma}$; So we can choose a sequence ${\alpha_j}$  satisfying the last property
                             and dense in $[0,1]$.
                               Remember that $\widehat{X^{\ast}} \subset Y_{\sigma}$, so Lemma \ref{lem:MAINLEM} implies that the set 
															\[
															\Lambda:=\bigcup\limits_j (\Gamma_{\alpha_j}+\alpha_j)
															\]
															is a US for $\mathrm{PW}_S$.
                            Clearly, the above $\Lambda$ is disjoint with $Q$, so Theorem \ref{customthm:MAIN} is proved.
                             Theorem \ref{thm:THM1} also follows.

\section{Proof of Theorem \ref{thm:THM2}}\label{sec:P2}
We begin with remarking that the construction below will give us a more refined structure of the set $Q$, the so called $\epsilon$-thin condition; See remark 2 in the last section for an explanation.

We will use the following easy lemma:
\begin{lem}\label{lem:poly}
Let $n>0$. Then, there exist $d =O(n^3)$, and a $1$-periodic trigonometric polynomial $P:[0,1] \rightarrow \mathbb{R}$ of degree $d$, such that $\left\|P\right\|_{L^2([0,1])}^2 \approx d$, $\left\|P\right\|_{L^{\infty}([0,1])} = O(d)$, and
	\[
	  \max\limits_{t \in I^c} \left|P(t)\right| \leq \frac{1}{n}
	\]
	Where $I = [\frac{1}{2}-\frac{1}{n},\frac{1}{2}+\frac{1}{n}]$ \end{lem}

\begin{proof} Take $P(t) = F_d(t+\frac{1}{2})$, where $F_d$ denotes the F\'ejer Kernel, and $d$ is large enough depending on $n$ as stated above. \end{proof}

Let us now turn to the main part of the proof, where we shall construct the $n$-th portion of the sets $S,Q$, and a function $f_n$ which will live in $\mathrm{PW}_{S_n}$ while having its $L^2$-norm localized on $Q_n$. 

 \begin{lem}\label{thm:MAIN2} 
For every $n \geq 1$ there exist a number $N$, sets $S_n,Q_n$ and a function $f_n$, such that:

\begin{itemize}
	\item  $S_n$ has a periodic gap:
	\[ 
	S_n \cap [k+0.4, k+0.6] = \emptyset \text{ for all } k \in \mathbb{Z}
	\]
	and 
	\[
	\left|S_n\right| \leq \frac{1}{2^n}.
	\]
	\item  $Q_n \subset [-N,N]$, and it satisfies, for every interval $I$ of length $\left|I\right|>\frac{1}{N}$,
	\[
	\left|Q_n \cap I\right| \leq \frac{1}{n} \left|I\right|.
	\]
	\item $f \in \mathrm{PW}_{S_n}$, and
	\[
	\int\limits_{Q_n^c} \left|f\right|^2 \leq \frac{1}{n} \int \left|f\right|^2.
	\]
\end{itemize}

 \end{lem}
Assuming Lemma \ref{thm:MAIN2}, let us prove the main theorem.
\begin{proof}[Proof of Theorem \ref{thm:THM2}] For each $n \geq 1$ we have $N,f_n,S_n,Q_n$ as in the lemma. 
Set $S = \bigcup\limits_n S_n$ and $Q = \bigcup\limits_n (Q_n+2N)$ (note that the translation of $Q_n$, does not affect the $L^2$ norm of $f_n$).
Clearly, $Q$ is of density $0$ (as the disjoint union of sets with density going to $0$), $S$ has periodic gap and measure at most $1$, and the series $\{f_n\}$ shows the pair $(S,Q)$ is \textsl{not} strongly annihilating.
 \end{proof}

\begin{proof}[Proof of Lemma \ref{thm:MAIN2}] 

We first describe the construction of the function. 
\begin{itemize}
\item Given $n$, use Lemma \ref{lem:poly} to obtain a trigonometric polynomial $P(t)$. Take some 'big' $N=N(n)$ (to be determined later).
	\item Let $P_N(t) = P(Nt)$. Denote $\mu = \widehat{P_N}$, then
	\begin{equation*}
	\mu = \sum\limits_{\left|k\right| \leq d} \widehat{P_N}(k) \delta(t-Nk).
	\end{equation*}
	Here, $\delta(\cdot)$ denotes the Dirac delta measure.
	\item Let $\psi \in C^{\infty}(\mathbb{R})$, $\text{Supp}(\psi) \subset [0,1]$, with $\left\|\psi\right\|_2 = 1$. Set 
	\[
	\varphi(t) = \psi(2^n(2d+1)t).
	\]
	\item Set $F = \mu \ast \varphi$, and let $f = \widehat{F}$; Then $f(t) = P_N(t) \widehat{\varphi}(t)$.
\end{itemize} 
The function $F$ is clearly supported on the set
\begin{equation*}
S_n = \bigcup\limits_{\left|j\right| \leq d} [jN,jN+\frac{1}{2^n(2d+1)}],
\end{equation*}
Which is the union of $2d+1$ intervals, each of length $\frac{1}{2^n(2d+1)}$.
Note that we have $\left|S_n\right| = \frac{1}{2^n}$; Moreover,
\[
S_n \cap [k+0.4, k+0.6] = \emptyset \text{ for all } k \in \mathbb{Z}.
\]
Let
\begin{equation*}
Q_n = \bigcup\limits_{j=-{N^2}+1}^{{N^2}-1} [\frac{j+\frac{1}{2}-\frac{1}{n}}{N}, \frac{j+\frac{1}{2}+\frac{1}{n}}{N}].
\end{equation*} 

Noting the following two facts, easily verifiable by a straightforward calculation:
\begin{itemize}
	\item Inside $(Q_n)^c \cap [-N,N]$, the $L^2$ norm of $f$ will be small because of Lemma \ref{lem:poly}: the function $f$ is concentrated on the intervals that are in $Q_n$; 		
	Specifically, we use the bound (for some absolute $c>0$):
	\[
\left|P_N \right| \leq \frac{c}{n} 
	\]
	and since the entire expression can be decomposed as
	\[
	\sum\limits_j\int\limits_{j}^{j+\frac{1}{N}} \left|P_N\right|^2 \left|\widehat{\varphi}(t)\right|^2 \mathrm{d}t
	\]	
	and on each chunk $\left|P_N \right|$ is small, the entire integral is small as well.
\item Outside $[-N,N]$, the $L^2$ norm of $f$ will be small because of the decay of $\varphi$, if $N$ is big enough (depending on $n$), while the size of $P_N$ can be bounded by \ref{lem:poly}; Specifically, we use the bounds (for some absolute $C_1,C_2>0$):
\begin{equation*}
\left|\widehat\psi(t)\right| \leq \frac{C_1}{1+t^2} \text{ and } \left\|P\right\|_{L^{\infty}([0,1])} \leq C_2d.
\end{equation*}
\end{itemize}
We get that: if $N = N(n)$ is chosen big enough, we have
\begin{equation*}
\int\limits_{(Q_n)^c} \left|f\right|^2 \leq \frac{C}{n} \left\|f\right\|_2^2
\end{equation*}
as required.
\end{proof}

\section{Remarks}
\begin{enumerate}
	\item   It seems interesting to compare Theorem \ref{thm:THM1}  with the result in \cite{NO},   
         where a set $S$ is constructed, of finite measure, such that its indicator function $1_S$ has the (closed)  				spectrum $Q$ of density zero.
         Notice that the set $S$ there does not have gaps.      
     \item In the proof of Theorem \ref{thm:THM2} it is not difficult, given $\epsilon>0$, to make the set $Q$ to be $\epsilon$-thin
          in the sense of \cite{Wolff}, which means, denoting $\rho(x) = \min(1,\frac{1}{\left|x\right|})$,
					\[
					\left|Q \cap [x-\rho(x),x+\rho(x)]\right| \leq 2\epsilon \rho(x) \text{ for every } x \in \mathbb{R}.
					\]
					In the notations of the proof, it will be $\frac{C}{n}$-thin for some absolute constant $C$.
					So, in the construction we can control the size of the gaps, compared to their distance from the origin.
					 In fact we can take care on all $\epsilon$ at once such that for any $\epsilon>0$,
					there exists some $M=M(\epsilon)$ such that $Q\cap \{\left|x\right|>M\}$ is $\epsilon$-thin.
\end{enumerate}

 \bibliographystyle{alpha}
  \bibliography{AnnihilationPaper}

\end{document}